\documentclass[11pt,a4paper]{amsart}
\usepackage{amssymb}
\usepackage{latexsym}
\usepackage{exscale}
\usepackage{amsfonts}
\usepackage{graphicx}
\usepackage{mathrsfs}
\usepackage{amsmath,amscd,amsthm}
\usepackage{bbm,pifont}
\usepackage{enumerate}
\usepackage{color}
\usepackage{amssymb}
\usepackage{latexsym}
\usepackage{exscale}
\usepackage[utf8]{inputenc}
\usepackage[T1]{fontenc}
\usepackage{dsfont}
\usepackage[mathscr]{eucal}

\allowdisplaybreaks

\DeclareMathAlphabet\mathcalbf{OMS}{cmsy}{b}{n}
\DeclareMathAlphabet\EuScript{U}{eus}{m}{n}
\DeclareMathAlphabet\EuScriptBold{U}{eus}{b}{n}

\numberwithin{equation}{section}
\newtheorem{theorem}{Theorem}[section]
\newtheorem{lemma}[theorem]{Lemma}

\def\C{\mathbb C}

\begin{document}
\allowdisplaybreaks

\title[Boundedness and compactness of Cauchy-type integral commutator]
{Boundedness and compactness of  Cauchy-type integral commutator on weighted Morrey spaces}

\author{Ruming Gong, Manasa N. Vempati, Qingyan Wu and Peizhu Xie}

\address{Ruming Gong, School of Mathematics and Information Science\\
         Guangzhou University\\
         Guangzhou, 510006 , China
         }
\email{gongruming@gzhu.edu.cn}

\address{Manasa N. Vempati,
Department of Mathematics,\\
Washington University -- St. Louis,\\
One Brookings Drive,
St. Louis, MO USA 63130-4899
}

\email{m.vempati@wustl.edu}

\address{Qingyan Wu, Department of Mathematics\\
         Linyi University\\
         Shandong, 276005, China
         }
\email{wuqingyan@lyu.edu.cn}

\address{Peizhu Xie, School of Mathematics and Information Science\\
         Guangzhou University\\
         Guangzhou, 510006 , China
         }
\email{xiepeizhu@gzhu.edu.cn}


\begin{abstract}
In this paper we study the boundedness and compactness characterizations of the commutator of  Cauchy type integrals $\EuScript C$  on a bounded strongly pseudoconvex domain $D$ in $\C^n$ with boundary $bD$ satisfying the minimum regularity condition $C^{2}$ based on the recent result of Lanzani--Stein and Duong--Lacey--Li--Wick--Wu. We point out that in this setting the Cauchy type integral $\EuScript C$ is the sum of the essential part $\EuScript C^\sharp$ which is a Calder\'on--Zygmund operator and a remainder $\EuScript R$ which is no longer a Calder\'on--Zygmund operator. We show that the commutator $[b, \EuScript C]$  is bounded on weighted Morrey space $L_{v}^{p,\kappa}(bD)$ ($v\in A_p, 1<p<\infty$)  if {\color{black}and only if}\ $b$ is in the BMO space on $bD$. Moreover, the commutator $[b, \EuScript C]$  is compact on weighted Morrey space $L_{v}^{p,\kappa}(bD)$ ($v\in A_p, 1<p<\infty$)  if {\color{black}and only if}\ $b$ is in the VMO space on $bD$.

\end{abstract}

\maketitle


\section{Introduction and statement of main results}
\setcounter{equation}{0}

\smallskip
Recently in the field of the interaction of complex analysis of several variables and harmonic analysis, people are interested in strongly pseudoconvex domains with boundaries satisfying the minimum regularity condition of class $C^2$, since some new phenomena emerge in the
 analysis and geometry of this kind of domains.
Lanzani and Stein \cite{LS}
 studied the Cauchy--Szeg\H o projection operator  in such kind of domains by introducing
 a family of Cauchy integrals $\{\EuScript C_\epsilon\}_\epsilon$, and established the $L^p(bD)$ ($1<p<\infty$, with respect to the Leray--Levi measure) boundedness  of $\EuScript C_\epsilon$.
 Different from the case of smooth strongly pseudoconvex domains, the kernels of these Cauchy integral operators do not satisfy the standard size or smoothness conditions for Calder\'on--Zygmund operators. It is important to study the harmonic analysis of this kind of Cauchy--Szeg\H o operators and Cauchy integral operators for the purpose of complex analysis, for example, to establish the theory of holomorphic Hardy spaces over such domains.
  Duong, Lacey, Li, Wick and the third author \cite{DLLWW}  proved the  boundedness and compactness for commutators of such Cauchy type integral operators. By abuse of notations, we omit the subscript $\epsilon$.


  \medskip\noindent
{\bf Theorem A (\cite{DLLWW}).} {\it
Suppose $D\subset \mathbb C^n$, $n\geq 2$, is  a bounded domain whose boundary is of class $C^2$ and is strongly pseudoconvex.
Suppose $b\in L^1(bD, d\lambda)$. Then for $1<p<\infty$,

$(1)$  $b\in{\rm BMO}(bD,d\lambda)$ if and only if the commutator $[b, \EuScript C]$ is bounded on  $L^p(bD, d\lambda)$.

$(2)$   $b\in{\rm VMO}(bD,d\lambda)$ if and only if the commutator $[b, \EuScript C]$ is compact on  $L^p(bD, d\lambda)$.

}

  \medskip

  Because the characterization of the boundedness and compactness of Calder\'on--Zygmund operator commutators on certain function spaces has many applications in various areas,
  such as harmonic analysis, complex analysis, (nonlinear) PDE, etc, it is an active direction recently to establish such a characterization for  singular integral operators (especially non-Calder\'on--Zygmund operators) over various function spaces. For example, it was proved for Calder\'on--Zygmund operator commutators on Morrey spaces over the Euclidean space by
  Di Fazio and Ragusa \cite{DiFazioRagusa91BUMIA} and Chen et al.\,\cite{CDW12CJM} , and necessary part by Komori and Shirai \cite{KS} on weighted Morrey space. The characterization for
  Cauchy integral and Beurling-Ahlfors transformation commutator on $\mathbb C$ over weighted Morrey space was established by Tao et al.\,\cite{TYY,TYY2}.
   In this paper, we extend the characterization of the boundedness and compactness of  Cauchy integral operator in Theorem A to weighted Morrey space.

\ Let $p\in (1,\infty)$.  A non-negative function $v\in L^1(bD)$ is in  $A_{p}(bD)$ if
\begin{align*}
[v]_{A_p(bD)}:=\sup_{B\subset bD}\left( \frac{1}{|B|}\int_{B}v(z)d\lambda(z)\right) \left( \frac{1}{|B|}%
\int_{B}v(z) ^{-1/(p-1)}d\lambda(z)\right) ^{p-1}<\infty,
\end{align*}
where the supremum is taken over all balls $B$ in $bD$.
A non-negative function $v\in L^1(bD)$ is in  $A_{1}(bD)$ if there exists a constant $C$
such that for all balls $B\subset bD$,
\begin{equation*}
\frac{1}{|B|}\int_{B}v( z) d\lambda(z)\leq C\mathop{\rm essinf}%
\limits_{x\in B}v(z) .
\end{equation*}%
For $p=\infty$, we define
\begin{equation*}
A_{\infty }(bD)= \bigcup_{1\leq p<\infty }A_{p}(bD).
\end{equation*}

Let $p\in(1,\infty)$, $\kappa\in(0,1)$ and $v\in A_p(bD)$.
The \emph{weighted Morrey space} $L_v^{p,\,\kappa}(bD)$ (c.f. \cite{Kok} )
 is defined by
\begin{equation*}
L_v^{p,\,\kappa}(bD):=\left\{f\in L_{loc}^{p}(bD):\,\,\|f\|_{L_v^{p,\,\kappa}(bD)}<\infty\right\}
\end{equation*}
with
\begin{equation*}
\|f\|_{L_v^{p,\,\kappa}(bD)}:=\sup_{B}
\left\{\frac{1}{[v(B)]^\kappa}\int_{B}|f(z)|^pv(z)\,d\lambda(z)\right\}^{1/p},
\end{equation*}
where
$$v(B)=\int_B v(z)d\lambda(z).$$
For the history of Morrey spaces one can refer to \cite{AX}.

The main result of our paper is to characterize the boundedness and compactness of  the commutator of Cauchy type integral  $\EuScript C$ on the weighted Morrey space, following the idea and approach in \cite{DLLWW}.



\begin{theorem}\label{cauchy}
Suppose $D\subset \mathbb C^n$, $n\geq 2$, is  a bounded domain whose boundary is of class $C^2$ and is strongly pseudoconvex.
Suppose $b\in L^1(bD),1<p<\infty$, $0<\kappa<1$ and $v\in A_p$. Then,
 $b\in{\rm BMO}(bD)$ if and only if the commutator $[b, \EuScript C]$ is bounded on  $L_{v}^{p,\kappa}(bD)$.
\end{theorem}

\begin{theorem}\label{vmo}
Suppose $D\subset \mathbb C^n$, $n\geq 2$, is  a bounded domain whose boundary is of class $C^2$ and is strongly pseudoconvex.
Suppose $b\in L^1(bD),1<p<\infty$, $0<\kappa<1$ and $v\in A_p$. Then,
 $b\in{\rm VMO}(bD)$ if and only if the commutator $[b, \EuScript C]$ is compact on  $L_{v}^{p,\kappa}(bD)$.
\end{theorem}

We also consider the Cauchy--Leray integral on a bounded domain in $\mathbb C^n$, which is strongly $\C$-linearly convex  and the boundary $bD$ satisfies the minimum regularity $C^{1,1}$  (for the details we refer to Section 3 below), such integral operators are studied by Lanzani and Stein in \cite{LS2014}.
They obtained the $L^p(bD)$ boundedness ($1<p<\infty$) of the  Cauchy--Leray transform $\mathcal C$ by showing that the kernel $K(w,z)$ of $\mathcal C$ satisfies the standard size and smoothness conditions of Calder\'on--Zygmund operators (for details of these definitions and notation, we refer the readers to Section 3), and
that  $\mathcal C$ satisfies a  suitable version of $T(1)$ theorem.  In \cite{DLLWW}, the authors also obtained the  boundedness and compactness for commutators of such transform.
Following a similar approach as in the proof of Theorem \ref{cauchy} and Theorem \ref{vmo}, we obtain the following results on  the Cauchy--Leray transform and its commutator.

\begin{theorem}\label{weight1}
Let $D$ be a bounded domain in $\mathbb C^n$ of class $C^{1,1}$ that is strongly $\mathbb C$-linearly convex. Let $1<p<\infty, v\in A_p$.  Then there exists a positive constant $C$ such that
\begin{align}\label{th1}
\|\mathcal C(f)\|_{L_{v}^{p}(bD)}\leq C \|f\|_{L_{v}^{p}(bD)}
\end{align}
holds for every function $f\in L_{v}^{p}(bD).$
\end{theorem}

\begin{theorem}\label{Cauchy-Leray}
Let $D$ be a bounded domain in $\mathbb C^n$ of class $C^{1,1}$ that is strongly $\mathbb C$-linearly convex and let $b\in L^1(bD),  1<p<\infty,  0<\kappa<1$ and $v\in A_p$.
Let $\mathcal C$ be the Cauchy--Leray transform $($as in \cite{LS2014}$)$. Then for $1<p<\infty$,

$(1)$  $b\in{\rm BMO}(bD)$ if and only if
the commutator $[b, \mathcal C]$ is bounded on  $L_{v}^{p, \kappa}(bD)$.

$(2)$  $b\in{\rm VMO}(bD)$ if and only if
the commutator $[b, \mathcal C]$ is compact on  $L_{v}^{p, \kappa}(bD)$.
\end{theorem}

\color{black}

This paper is organised as follows.  In Section 2 we recall the notation and definitions related to a family of Cauchy integrals for bounded strongly pseudoconvex domains in $\C^n$ with minimal smoothness, then we prove Theorems \ref{cauchy} and \ref{vmo}.
 In Section 3 we recall the notation and definitions related to the Cauchy-Leray integral for bounded $\C$-linearly convex domains in $\mathbb C^n$ with minimal smoothness and give the proof of Theorem \ref{weight1} and Theorem \ref{Cauchy-Leray}.


\section{Commutator of Cauchy type integral for bounded strongly pseudoconvex domains with minimal smoothness}
\setcounter{equation}{0}

In this section, we
always assume that $D$ is a bounded strongly pseudoconvex
domain {whose boundary is of class $C^2$.

Let ${\tt d}(w,z)$ be the quasidistance on the boundary $bD$, which is defined as in \cite[Section 2.3] {LS} and satisfies the following conditions: there exist constants $A_1>0$ and  $C_d>1$ such that for all $w,z,z'\in bD$,
\begin{equation}\label{metric d}
\begin{split}
& {\tt d}(w,z)=0\quad {\rm iff}\quad w=z;\\
 & {\tt d}(w,z)\leq C_d\big( {\tt d}(w,z') +{\tt d}(z',z)\big);\\
 &A_1^{-1} {\tt d}(z,w)\leq  {\tt d}(w,z) \leq A_1 {\tt d}(z,w).
\end{split}
\end{equation}
Let $d\lambda$ be the Leray--Levi measure on $bD$ (c.f. \cite[138]{LS}).
\smallskip
Then one has
$$   d\lambda(w)=\Lambda(w)d\sigma(w),
 $$
where  $d\sigma$ is the induced Lebesgue measure on $bD$ and $\Lambda(w)$ is a continuous function such that
$ c\leq \Lambda(w)\leq \tilde c, w\in bD$, with $c$ and $\tilde c$ two positive constants.
\color{black}

We also recall the  boundary balls $ B_r(w) $ determined via the quasidistance ${\tt d }$, i.e.,
\begin{align}\label{ball}
B_r(w) :=\{ z\in bD:\ {\tt d}(z,w)<r \}, \quad {\rm where\ } w\in bD.
\end{align}
According to \cite[p. 139]{LS}, we have
\begin{align}\label{lambdab}
c_\lambda^{-1} r^{2n}\leq \lambda\big(B_r(w) \big)\leq c_\lambda r^{2n},\quad 0<r\leq 1,
\end{align}
for some $c_\lambda>1$.

In \cite{LS}, the authors defined a family of Cauchy integrals $\{\EuScript C_\epsilon\}_\epsilon$ and studied their properties when $\epsilon$ is kept fixed. For convenience of notation we will henceforth drop explicit reference to $\epsilon$.  To study the Cauchy  transform $\EuScript C$,  which is the restriction of such a Cauchy integral on $bD$,
one of the key steps in  \cite{LS} is that they provided a constructive decomposition of $ \EuScript C$ as follows:
$$   \EuScript C =  \EuScript C^\sharp+ \EuScript R,  $$
where the essential part
\begin{align}
 \EuScript C^\sharp(f) (z) := \int_{w\in bD}  C^\sharp(w,z) f(w) d\lambda(w), \quad z\in bD
\end{align}
and the reminder
$$\EuScript R(f)(z):=\int_{w\in bD}R(w,z)f(w)d\lambda(w).$$
Thus, if we write
$$\EuScript C(f)(z):=\int_{w\in bD}C(w,z)f(w)d\lambda(w).$$
Then
$$C(w,z)=C^\sharp(w,z)+ R(w,z),$$
where the kernel $C^\sharp(w,z)$ satisfies the standard size and smoothness conditions
for Calder\'on--Zygmund operators, i.e. there exists a positive constant $A_2$ such that for every $w,z\in bD$ with $w\not=z$,
\begin{equation}\label{gwz}
\begin{split}
& |C^\sharp(w,z)|\leq A_2 {\displaystyle1\over\displaystyle{\tt d}(w,z)^{2n}};\\
&|C^\sharp(w,z) - C^\sharp(w,z')|\leq A_2  {\displaystyle {\tt d}(z,z')\over \displaystyle {\tt d}(w,z)^{2n+1} },\quad {\rm if}\ {\tt d}(w,z)\geq c {\tt d}(z,z');\\
&|C^\sharp(w,z) - C^\sharp(w',z)|\leq A_2 {\displaystyle {\tt d}(w,w')\over \displaystyle {\tt d}(w,z)^{2n+1} },\quad {\rm if}\ {\tt d}(w,z)\geq c{\tt d}(w,w')
 \end{split}
\end{equation}
for an appropriate constant $c>0$.
 However, the kernel $R(w,z)$ of $\EuScript R$ satisfies a size condition and a smoothness condition for only one of the variables as follows: there exists
a positive constant $C_R$ such that for every $w,z\in bD$ with $w\not=z$,
\begin{equation}\label{cr}
                \begin{split}
&|R(w,z)|\leq C_R {\displaystyle1\over \displaystyle{\tt d}(w,z)^{2n-1}};\\
& |R(w,z)-R(w,z')|\leq C_R {\displaystyle{\tt d}(z,z')\over\displaystyle {\tt d}(w,z)^{2n}},\quad {\rm if\ }  {\tt d}(w,z)\geq c_R {\tt d}(z,z')
                \end{split}
\end{equation}
for an appropriate large constant $c_R$.

We also denote by $BUC(bD)$ the space of all bounded uniformly continuous functions on $bD$.
We first point out that the Leray--Levi measure $d\lambda$ on $bD$ is a doubling measure, and  satisfies
the condition (1.1) in \cite{KL2}.
{
\begin{lemma}[\cite{DLLWW}]\label{measure lambda}
The Leray--Levi measure $d\lambda$ on $bD$ is doubling, i.e., there is a positive constant $C$ such that for all $x\in bD$ and $0<r\leq1$,
$$ 0<\lambda(B_{2r}(x))\leq C\lambda(B_{r}(x))<\infty. $$
Moreover, $\lambda$
satisfies the condition: there exist a constant $\epsilon_0\in(0,1)$ and a positive constant $C$ such that
$$ \lambda( B_r(x)\backslash B_r(y) ) +  \lambda( B_r(y)\backslash B_r(x) ) \leq C\bigg( { {\tt d}(x,y) \over r}  \bigg)^{\epsilon_0}   $$
for all $x,y\in bD$ and ${\tt d}(x,y)\leq r\leq1$.
\end{lemma}
}

 We now recall the BMO space on $bD$.
 Consider $(bD, {\tt d}, d\lambda)$ as a space of homogeneous type with $bD$ compact.
Then ${\rm BMO }(bD)$ is defined as the set of all $b\in L^1(bD)$ such that
$$ \|b\|_*:=\sup_{ B\subset bD} {1\over\lambda(B)}\int_{B} |b(w)-b_B|d\lambda(w)<\infty, $$
where
\begin{align}\label{fb}
b_B={1\over \lambda(B)}\int_B b(z)d\lambda(z).
\end{align}
And the norm is defined as
$$\|b\|_{{\rm BMO }(bD)}:=\|b\|_*+ \|b\|_{L^1(bD)}. $$


The maximal function $Mf$ is defined as
$$Mf(z)=\sup_{z\in B\subset bD}{1\over\lambda(B)}\int_B |f(w)|d\lambda(w).$$
The sharp function $f^\#$ is defined as
$$f^\#(z)=\sup_{z\in B\subset bD}{1\over\lambda(B)}\int_B |f(w)-f_B|d\lambda(w),$$
where $f_B$ is defined in \eqref{fb}.

{Note that from  Lemma \ref{measure lambda}, $d\lambda$ is a doubling measure. Hence, we have the  following results.}

\begin{lemma}[\cite{Kok}]\label{lv}
Let $v\in A_p(bD)$, $p\geq 1$. Then there exist constants $C, \sigma>0$ such that for every ball $B$ and measurable subset $E\subset B$ the inequality
$$\frac{v(E)}{v(B)}\leq C\Big(\frac{\lambda(E)}{\lambda(B)} \Big)^\sigma$$
holds.
\end{lemma}

\begin{lemma}[\cite{MS}]\label{maximal}
Let $v\in A_p(bD)$, $1< p<\infty$. There exists a constant $C$ such that for every $f\in L_v^p(bD)$,
\begin{align*}
\left\|Mf\right\|_{L_v^p(bD)}\leq C\|f\|_{L_v^p(bD)},
\end{align*}
where $\|f\|_{L_v^p(bD)}^{p}=\int_{bD} |f(z)|^pv(z)d\lambda(z)$.
\end{lemma}

\begin{lemma}[\cite{PS}]\label{lemma-sharp}
Let $v\in A_p(bD)$, $1< p<\infty$. There exists a constant $C$ such that if $\|f\|_{L_v^p(bD)}<\infty$, then
\begin{align*}
\left\|f\right\|_{L_v^p(bD)}^{p}\leq C\Big(v(bD)(f_{bD})^p+\|f^\#\|_{L_v^p(bD)}^{p}\Big).
\end{align*}
\end{lemma}



\begin{lemma}[\cite{Kr}]\label{lem-jn1}
If $f\in {\rm BMO}(bD)$, then there exist positive constants $C_1$ and $C_2$ such that for every ball $B\subset bD$ and every $\alpha>0$, we have
$$\lambda(\{x\in B: |f(x)-f_B|>\alpha \})\leq C_1\lambda(B)\exp\Big\{- {C_2\over \|f\|_{{\rm BMO}(bD)}}\alpha\Big\}.$$
\end{lemma}

According to \cite[Theorem 5.5]{HT}, we have the following result for BMO functions on $bD$.

\begin{lemma}\label{bmoqq}
Let $0<p<\infty$, $v\in A_\infty(bD)$, $f\in {\rm{BMO}}(bD)$. Then
$$\|f\|_{{\rm{BMO}}(bD)}\approx\sup_{B\subset bD}\bigg\{{1\over v(B)}\int_B\big| f(z)-f_{B,v} \big|^pv(z)d\lambda(z)  \bigg\}^{1\over p},$$
where $f_{B,v}={1\over v(B)}\int_{B}f(z)v(z)d\lambda(z).$
\end{lemma}



\subsection{Characterisation of ${\rm BMO}(bD)$ via the Commutator $[b, \EuScript C]$ }

\begin{proof}[Proof of necessity of  Theorem \ref{cauchy}]

\smallskip

We first prove necessity, namely that $b\in {\rm BMO}(bD)$ implies the boundedness of $[b, \EuScript C]$.

We can write
$$[b,\EuScript C]=[b,\EuScript C^\sharp]+[b,\EuScript R].$$

Since the kernel of $\EuScript C^\sharp$ is a standard kernel on $bD\times bD$, according to \cite[Theorerm 3.4]{KS},  we can obtain that $[b, \EuScript C^\sharp]$ is bounded on  $L_v^{p,\kappa}(bD)$ and
\begin{align*}
\|[b,\EuScript C^\sharp]\|_{L_v^{p,\kappa}(bD)\rightarrow L_v^{p,\kappa}(bD)}\lesssim \|b\|_{{\rm BMO} (bD)}.
\end{align*}
Thus, it suffices to show that
\begin{align}\label{[b,R] bounded}
\|[b,\EuScript R]\|_{L_v^{p,\kappa}(bD)\rightarrow L_v^{p,\kappa}(bD)}\lesssim \|b\|_{{\rm BMO} (bD)}.
\end{align}

Now fix a ball $B=B_r(z_0)\subset bD$ and decompose $f=f\chi_{bD\cap 2B}+f\chi_{bD\setminus 2B}=:f_1+f_2$. Then
\begin{align*}
&{1\over v(B)^{\kappa}}\int_B\left|[b,\EuScript R] f(z)  \right|^pv(z)d\lambda(z)\\
&\lesssim \bigg\{ {1\over v(B)^{\kappa}}\int_B\left|[b,\EuScript R] f_1(z)  \right|^pv(z)d\lambda(z)+{1\over v(B)^{\kappa}}\int_B\left|[b,\EuScript R] f_2(z)  \right|^pv(z)d\lambda(z)\bigg\}\\
&=:I+II.
\end{align*}

For the term $I$, by the proof of Theorem 1.6 in \cite{DLLW}, we find
\begin{align*}
{1\over v(B)^{\kappa}}\int_B\left|[b,\EuScript R] f_1(z)  \right|^pv(z)d\lambda(z)
&\leq{1\over v(B)^{\kappa}} \int_{bD}\left|[b,\EuScript R] f_1(z)\right|^pv(z)d\lambda(z)\\
&\lesssim \|b\|^p_{ {\rm BMO}(bD) }{1\over v(B)^{\kappa}}\int_{2B}|f(z)|^pv(z)d\lambda(z)\\
&\lesssim\|b\|^p_{ {\rm BMO}(bD) }\|f\|^p_{L_v^{p,\,\kappa}(bD)}.
\end{align*}

For the term $II$, observe that for $z\in B$, by \eqref {cr}, we have
\begin{align*}
\left|[b,\EuScript R] f_2(z)\right|^p
&\leq\bigg(\int_ {bD}|b(z)-b(w)| |R(w,z)| |f_2(w)|d\lambda(w)\bigg)^p\\
&\lesssim \bigg( \int_{ bD\setminus 2B}{|b(z)-b(w)|\over{\tt d}(w, z)^{2n-1}} |f(w)|d\lambda(w)\bigg)^p\\
&\lesssim \bigg(\int_{ bD\setminus 2B}{ |f(w)|\over{\tt d}(w, z_0)^{2n-1}}\left\{\left|b(z)-b_{B,v}\right|+\left|b_{B,v}-b(w)\right|\right\}d\lambda(w)\bigg)^p\\
&\lesssim\bigg(\int_{ bD\setminus 2B}{ |f(w)|\over{\tt d}(w, z_0)^{2n-1}}d\lambda(w)  \bigg)^p \left|b(w)-b_{B,v}\right|^p\\
&\quad+\bigg(\int_{ bD\setminus 2B}{ |f(w)|\over{\tt d}(w, z_0)^{2n-1}}\left|b_{B,v}-b(w)\right|d\lambda(w) \bigg)^p,
\end{align*}
where
$b_{B,v}={1\over v(B)}\int_{B}b(z)v(z)d\lambda(z).$ Then by using that $bD$ is bounded we can obtain
\begin{align*}
II&={1\over v(B)^\kappa}\int_B\left|[b,\EuScript R] f_2(z)\right|^pv(z)d\lambda(z)\\
&\lesssim
{1\over v(B)^\kappa} \bigg(\int_{bD\setminus 2B}{ |f(w)|\over{\tt d}(w, z_0)^{2n-1}}d\lambda(w) \bigg)^p\int_B\left|b(z)-b_{B,v}\right|^pv(z)d\lambda(z)  \\
&\quad+\bigg(\int_{ bD\setminus 2B}{ |f(w)|\over{\tt d}(w, z_0)^{2n-1}}\left|b_{B,v}-b(w)\right|d\lambda(w) \bigg)^pv(B)^{1-\kappa}\\
&=:{II}_1+{II}_2.
\end{align*}

For ${II}_1$, by the H\"older inequality, Lemma \ref{bmoqq} and Lemma \ref{lv}, we have
\begin{align*}
II_1&\lesssim \|f\|^p_{L_{v}^{p,\kappa}(bD)}{1\over v(B)^\kappa}\bigg(\sum_{j=1}^{\infty} {1\over v(2^{j+1}B)^{{1-\kappa\over p}}} \bigg)^p\int_B\left|b(z)-b_{B,v}\right|^pv(z)d\lambda(z)\\
&\lesssim\|b\|_{{\rm BMO}(bD)}\|f\|^p_{L_{v}^{p,\kappa}(bD)}\bigg(\sum_{j=1}^{\infty} {v(B)^{1-\kappa\over p}\over v(2^{j+1}B)^{{1-\kappa\over p}}} \bigg)^p\\
&\lesssim\|b\|_{{\rm BMO}(bD)}\|f\|^p_{L_{v}^{p,\kappa}(bD)}.
\end{align*}

For ${II}_2$, by the H\"older inequality, we have
\begin{align*}
II_2&\lesssim v(B)^{1-\kappa}\bigg(\sum_{j=1}^{\infty}{1\over \lambda(2^jB)}\int_{2^{j+1}B}|f(w)| \left| b(w)-b_{B,v}\right| d\lambda(w) \bigg)^p \\
&\lesssim v(B)^{1-\kappa}\bigg\{\sum_{j=1}^{\infty}{1\over \lambda(2^jB)}\left(\int_{2^{j+1}B}|f(w)|^p v(w)d\lambda(w)\right)^{1\over p}\\
&\qquad \times\left(\int_{2^{j+1}B} \left| b(w)-b_{B,v}\right| ^{p'}v(w)^{1-p'}d\lambda(w)\right)^{1\over p'} \bigg\}^p\\
 &\lesssim v(B)^{1-\kappa}\|f\|^p_{L^{p,\kappa}(bD)}\bigg\{\sum_{j=1}^{\infty}{v(2^{j+1}B)^{\kappa\over p}\over \lambda(2^jB)}
\left(\int_{2^{j+1}B} \left| b(w)-b_{B,v}\right| ^{p'}v(w)^{1-p'}d\lambda(w)\right)^{1\over p'} \bigg\}^p\\
 &\lesssim v(B)^{1-\kappa}\|f\|^p_{L^{p,\kappa}(bD)}\bigg\{\sum_{j=1}^{\infty}{v(2^{j+1}B)^{\kappa\over p}\over \lambda(2^jB)}
\bigg[\left(\int_{2^{j+1}B} \left| b(w)-b_{2^{j+1}B,v^{1-p'}}\right| ^{p'}v(w)^{1-p'}d\lambda(w)\right)^{1\over p'}\\
&\qquad+\left(\int_{2^{j+1}B} \left| b_{2^{j+1}B,v^{1-p'}}-b_{B,v}\right| ^{p'}v(w)^{1-p'}d\lambda(w)\right)^{1\over p'} \bigg\}^p\\
&=: v(B)^{1-\kappa}\|f\|^p_{L^{p,\kappa}(bD)}\bigg[\sum_{j=1}^{\infty}{v(2^{j+1}B)^{\kappa\over p}\over \lambda(2^jB)}\left({II}_{21}+{II}_{22} \right)\bigg]^p
\end{align*}


For ${II}_{21}$, since $v\in A_p(bD)$, we have $v^{1-p'}\in A_{p'}(bD)$, where $1/p+1/p'=1$. By Lemma \ref{bmoqq}, we can obtain that
$${II}_{21}\lesssim\|b\|_{{\rm {BMO}}(bD)}\left[v^{1-p'}(2^{j+1}B)\right]^{1\over p'}.$$

For ${II}_{22}$,  by Lemma \ref{bmoqq}, we have
\begin{align*}
\left|b_{2^{j+1}B, v^{1-p'}}-b_{B,v}\right|&\leq \left|b_{2^{j+1}B, v^{1-p'}}-b_{2^{j+1}B}\right|+\left|b_{2^{j+1}B}-b_{B}\right|
+\left|b_{B}-b_{B,v}\right|\\
&\leq{1\over v^{1-p'}(2^{j+1}B)}\int_{2^{j+1}B}\left|b(w)-b_{2^{j+1}B}\right|v(w)^{1-p'}d\lambda(w)\\
& +2^{2n}(j+1)\|b\|_{{\rm {BMO}}(bD)}
+{1\over v(B)}\int_{B}\left|b(w)-b_{B}\right|v(w)d\lambda(w).
\end{align*}

Since $b\in {\rm {BMO}}(bD)$, by Lemma \ref{lv} and Lemma  \ref{lem-jn1}, there exist $\bar C_1>0$ and $\bar C_2>0$ such that for any ball $B$ and
$\alpha>0$,
$$v\left( \left\{g\in B: |b(z)-b_B|>\alpha \right\}\right)\leq \bar C_1v(B)e^{-{\bar C_2\alpha\sigma\over\|b\|_{{\rm {BMO}}(bD)}}}, $$
for some $\sigma>0$. Therefore,
\begin{align*}
\int_B |b(w)-b_B|v(w)d\lambda(w)&=\int_{0}^{\infty}v\left(\{z\in B: |b(z)-b_B|>\alpha \} \right)d\alpha\\
&\lesssim v(B)\int_{0}^{\infty}e^{-{\bar C_2\alpha\sigma\over\|b\|_{{\rm {BMO}}(bD)}}} d\alpha\\
&\lesssim v(B)\|b\|_{{\rm {BMO}}(bD)}.
\end{align*}
Similarly, we have
$$\int_{2^{j+1}B}\left|b(w)-b_{2^{j+1}B}\right|v(w)^{1-p'}d\lambda(w)\lesssim (j+1)\|b\|_{{\rm {BMO}}(bD)}v^{1-p'}(2^{j+1}B).$$
Thus,
$${II}_{22}\lesssim (j+1)\|b\|_{{\rm {BMO}}(bD)}\left[v^{1-p'}(2^{j+1}B)\right]^{1\over p'}$$
Now together with Lemma \ref {lv}, we have
\begin{align*}
{II}_2&\lesssim  v(B)^{1-\kappa}\|b\|^p_{{\rm {BMO}}(bD)}\|f\|^p_{L_{v}^{p,\kappa}(bD)}\bigg\{\sum_{j=1}^{\infty}{v(2^{j+1}B)^{\kappa\over p}\over \lambda(2^jB)}(j+1)\big[v^{1-p'}(2^{j+1}B)\big]^{1/p'}
\bigg\}^p\\
 &\lesssim \|b\|^p_{{\rm {BMO}}(bD)}\|f\|^p_{L_{v}^{p,\kappa}(bD)}\bigg[\sum_{j=1}^{\infty}{(j+1)v(B)^{1-k\over p}\over v(2^{j+1}B)^{1-\kappa\over p}}
 \bigg]^p\\
 &\lesssim \|b\|^p_{{\rm {BMO}}(bD)}\|f\|^p_{L_{v}^{p,\kappa}(bD)}\bigg[\sum_{j=1}^{\infty}(j+1) 2^{-(j+1)(1-\kappa)Q\sigma\over p}
 \bigg]^p\\
&\lesssim \|b\|^p_{{\rm {BMO}}(bD)}\|f\|^p_{L_{v}^{p,\kappa}(bD)}.
\end{align*}
\color{black}
Therefore,
$$II\lesssim \|b\|^p_{{\rm {BMO}}(bD)}\|f\|^p_{L_{v}^{p,\kappa}(bD)}.$$
Consequently, we obtain
$$\|[b,\EuScript R]f\|_{L_v^{p,\kappa}(bD)}\lesssim \|b\|_{{\rm BMO} (bD)}\|f\|^p_{L_{v}^{p,\kappa}(bD)}.$$
This completes the proof of  the necessity part.
\end{proof}

In order to prove the sufficiency of Theorem \ref{cauchy}, we need the following lemma.

\begin{lemma}[\cite{DLLWW}]\label{nc}
Denote by $ C_1(w,z)$ and $ C_2(w,z)$
the real and imaginary parts of $ C(w,z)$, respectively.
Then there is  at least one of the $C_i$ above satisfies the following argument:

There exist positive constants $\gamma_0, A$ such that for every ball $B=B_r(z_0)\subset bD$ with $r<\gamma_0$, there exists another ball
$\tilde B = B_r(w_0)\subset bD$ with  $Ar \leq {\tt d}(w_0,z_0) \leq (A+1)r$ such that
for every $z\in B$ and $w\in \tilde B$, $ C_i(w,z)$ does not change sign and $$| C_i(w,z)|\geq { c\over  {\tt d}(w,z)^{2n} }.$$

\end{lemma}

\begin{proof}[Proof of sufficiency of Theorem \ref{cauchy}]

\smallskip

We  turn to prove the sufficient condition, namely that if $[b,\EuScript C]$ is bounded on $L_v^{p,\kappa}(bD)$, then $b\in {\rm BMO}(bD)$. We mainly follow the method and technique in \cite{DLLWW}.

Assume that $b$ is in $L^1(bD)$ and that $\left\|[b,\EuScript C]\right\|_{L_v^p(bD)\to L_v^p(bD)} <\infty$.
Let $\gamma_0$ be the constant in Lemma \ref{nc}. We test the ${\rm BMO}(bD,d\lambda)$ condition on the case of balls with big radius and small radius.

Case 1: In this case we work with balls with a large radius, $r\geq \gamma_0$.

By \eqref{lambdab} and by the fact that
$\lambda(B)\geq \lambda( B_{\gamma_0}(z_0)) \approx \gamma_0^{2n}$, we obtain that
\begin{align*}
{1\over \lambda(B)} \int_B|b(w)-b_B|d\lambda(w)\lesssim \gamma_0^{-2n} \|b\|_{L^1(bD, d\lambda)}.
\end{align*}

Case 2: In this case we work with balls with a small radius, $r<\gamma_0$.

We aim to prove that for every fixed ball $B=B_r(z_0)\subset bD$ with radius $r<\gamma_0$,
\begin{align*}
{1\over \lambda(B)} \int_B|b(w)-b_B|d\lambda(w)\lesssim \left\|[b,\EuScript C]\right\|_{L_v^{p,\kappa}(bD)\to L_v^{p,\kappa}(bD)}.
\end{align*}

Now let $\tilde B=B_r(w_0)$ be the ball chosen as Lemma \ref{nc}, and let $m_b(\tilde B)$ be the median value of $b$ on the ball $\tilde B$ with respect to the measure $d\lambda$.

Following the sufficiency proof of Theorem 1.1 (1) in \cite{DLLWW},
 we can choose sets $F_{1}, F_{2}, E_{1}, E_{2}$ such that $\tilde B=F_{1}\cup F_{2}$, $B=E_{1}\cup E_{2}$, $E_{1}\cap E_{2}=\emptyset$,
\begin{align}\label{f1f2-1 1}
\lambda(F_{i}) \geq{\lambda(\tilde B)\over 2},\quad i=1,2.
\end{align}
and
\begin{align*}
{1\over \lambda(B)}\int_{E_{1}}\big|b(z)-m_b(\tilde B)\big|d\lambda(z)\lesssim {1\over \lambda(B)}\int_{E_{1}}\left|[b, \EuScript C](\chi_{F_1})(z)\right|d\lambda(z).
\end{align*}
\color{black}
Then, by H\"older's inequality, $v\in A_p$ and the fact that $\left\|[b,\EuScript C]\right\|_{L_v^{p,\kappa}(bD)\to L_v^{p,\kappa}(bD)} <\infty$, we can obtain
\begin{align*}
&{1\over \lambda(B)}\int_{E_{1}}\big|b(z)-m_b(\tilde B)\big|d\lambda(z)\\
&\lesssim
{1\over \lambda(B)}\left(\int_{E_1}v^{-{p'\over p}}d\lambda(z)\right)^{1\over p'}\bigg(\int_{E_{1}}\left|[b, \EuScript C](\chi_{F_1})(z)\right|^pv(z)d\lambda(z)\bigg)^{1\over p}\\
&\lesssim
 {1\over \lambda(B)}\left(\int_{B}v^{-{p'\over p}}d\lambda(z)\right)^{1\over p'}(v(B))^{\kappa\over p}\|[b, \EuScript C]\chi_{F_1}\|_{L_v^{p,\kappa}(bD)}\\
 &\lesssim
 \left(v(B)\right)^{{\kappa-1\over p}}
\|[b, \EuScript C]\|_{L_v^{p,\kappa}(bD)\to L_v^{p,\kappa}(bD)}\|\chi_{F_1}\|_{L_v^{p,\kappa}(bD)}\\
 &\lesssim
 \left(v(B)\right)^{{\kappa-1\over p}} \left(v(\tilde B)\right)^{{1-\kappa\over p}}
\|[b, \EuScript C]\|_{L_v^{p,\kappa}(bD)\to L_v^{p,\kappa}(bD)}\\
&\lesssim
  \|[b, \EuScript C]\|_{L_v^{p,\kappa}(bD)\to L_v^{p,\kappa}(bD)}.
\end{align*}

Similarly, we can obtain that
\begin{align*}
{1\over \lambda(B)}\int_{E_{2}}\big|b(z)-m_b(\tilde B)\big|d\lambda(z)
\lesssim
 \|[b, \EuScript C]\|_{L_v^{p,\kappa}(bD)\to L_v^{p,\kappa}(bD)}.
\end{align*}

Consequently,
\begin{align*}
{1\over \lambda(B)}\int_{B}\big|b(z)-m_b(\tilde B)\big|d\lambda(z)\lesssim
 \|[b, \EuScript C]\|_{L_v^{p,\kappa}(bD)\to L_v^{p,\kappa}(bD)}.
\end{align*}

Therefore,
\begin{align*}
{1\over \lambda(B)}\int_{B}\big|b(z)-b_B\big|d\lambda(z)
\leq {2\over \lambda(B)}\int_{B}\big|b(z)-m_b(\tilde B)\big|d\lambda(z)
\lesssim   \|[b, \EuScript C]\|_{L_v^{p,\kappa}(bD)\to L_v^{p,\kappa}(bD)}.
\end{align*}
This finishes the proof of sufficiency of Theorem \ref{cauchy}.
\end{proof}

\subsection{Characterisation of ${\rm VMO}(bD,d\lambda)$ via the Commutator $[b, \EuScript C]$}

Based on Lemma \ref{measure lambda}, we have the following  fundamental lemma from \cite[Lemma 1.2]{KL2}. 
\begin{lemma}\label{lemma 1 KL2}
Let $b\in{\rm VMO}(bD, d\lambda)$. Then for any $\xi>0$, there is a function $b_\xi\in BUC(bD)$ such that
\begin{align}\label{f_eta -f}
\|b_\xi -b\|_*<\xi.
\end{align}
Moreover, $b_\xi$ satisfies the following conditions: there is an $\epsilon\in (0,1)$ such that
\begin{align}\label{f_eta}
|b_\xi(w) -b_\xi(z)|<C_\xi {\tt d}(w,z)^\epsilon,\quad \forall w,z\in bD.
\end{align}
\end{lemma}
For each $0<\eta<<1$, we let
$R^\eta(w,z)$ be a continuous extension of the kernel $R(w,z)$ of $\EuScript R$ from $bD\times bD\backslash \{(w,z): {\tt d}(w,z)<\eta\}$ to $bD\times bD$ such that
\begin{align*}
&R^\eta(w,z) =R(w,z),\quad {\rm if}\ \ {\tt d}(w,z)\geq\eta;\\
&|R^\eta(w,z)|\lesssim {1\over {\tt d}(w,z)^{2n-1}}, \quad {\rm if}\ \  {\tt d}(w,z)<\eta;\\
& R^\eta(w,z)=0, \quad {\rm if}\ \  {\tt d}(w,z)<\eta/c\ \ {\text {for some}}\ \ c>1.
\end{align*}
Let $\EuScript R^\eta$ be the integral operator associate to the kernel $R^\eta(w,z)$.
Then we have the following approximation result.
\begin{lemma}\label{lemma 2 KL2}
Let $b\in BUC(bD)$ satisfy
\begin{align}\label{f_eta 1}
|b(w) -b(z)|<C_\eta {\tt d}(w,z)^\epsilon,\quad {\rm\ for\ some\ } C_\eta\geq1,\ \epsilon\in(0,1),\ \forall w,z\in bD.
\end{align}
Then for $1<p<\infty$, $0<\kappa<1$ and $v\in A_p$, we have
$$ \|[b, \EuScript R]-[b, \EuScript R^\eta]\|_{L_v^{p,\kappa}(bD)\to L_v^{p,\kappa}(bD)}\to0 $$
as $\eta\to0$.
\end{lemma}
\begin{proof}
Let $f\in L_v^{p,\kappa}(bD)$. For any $z\in bD$, we have

\begin{align*}
 &\Big|[b, \EuScript R]f(z)-[b, \EuScript R^\eta]f(z)\Big|\\
 &= \bigg|\int_{{\eta \over c}\leq {\tt d}(w,z)<\eta}\Big(b(z)-b(w)\Big)R^\eta(w,z)f(w)d\lambda(w) \\
 &\quad- \int_{ {\tt d}(w,z)<\eta}\Big(b(z)-b(w)\Big)R(w,z)f(w)d\lambda(w)   \bigg|\\
&\lesssim \int_{ {\tt d}(w,z)<\eta}{{\tt d}(w,z)^\epsilon\over {\tt d}(w,z)^{2n-1}}|f(w)|d\lambda(w)\\
&\lesssim \sum_{j=0}^{\infty}\int_{ {\eta\over 2^{j+1}}\leq{\tt d}(w,z)<{\eta\over 2^{j}}}{{\tt d}(w,z)^{\epsilon+1}\over {\tt d}(w,z)^{2n}}|f(w)|d\lambda(w)\\
&\lesssim \sum_{j=0}^{\infty}{\eta^{\epsilon+1}\over 2^{(\epsilon+1)j}}{1\over \lambda\big(B_{{\eta\over 2^{j}}}(z) \big)}\int_{ {\tt d}(w,z)<{\eta\over 2^{j}}}|f(w)|d\lambda(w)\\
&\lesssim \eta^{\epsilon+1}Mf(z).
\end{align*}
Then we have

$$ \|[b, \EuScript R]f-[b, \EuScript R^\eta]f\|_{L_v^{p,\kappa}(bD)}\lesssim \eta^{\epsilon+1} \|f\|_{L_v^{p,\kappa}(bD)}. $$

This implies that

$$\lim_{\eta\to0} \|[b, \EuScript R]-[b, \EuScript R^\eta]\|_{L_v^{p,\kappa}(bD)\to L_v^{p,\kappa}(bD)}=0.$$
\end{proof}

We now prove  Theorem \ref{vmo}.

\medskip

\begin{proof}[Proof of   Theorem \ref{vmo}]
{\bf Sufficiency:}
Assume that $v\in A_p, 1<p<\infty$ and that $[b, \EuScript C]$ is compact on $L_v^{p,\kappa}(bD)$, then $[b,\EuScript C]$ is bounded on $L_v^{p,\kappa}(bD)$. By  Theorem \ref{cauchy}, we have $b\in {\rm BMO}(bD)$. Without loss of generality, we may assume that $\|b\|_{  {\rm BMO}(bD) }=1$.

\textcolor{black}{
To show $b\in {\rm VMO}(bD)$, we  seek a contradiction.  In its simplest form, the contradiction is that there is no bounded operator $T \;:\; \ell^{p} (\mathbb N) \to \ell^{p} (\mathbb N)$ with $Te_{j } = T e_{k} \neq 0$ for all $j,k\in \mathbb N$.  Here, $e_{j}$ is the
standard basis for $\ell^{p} (\mathbb N)$.  }

\textcolor{black}{
The main step is to construct the approximates to a standard basis in $\ell^{p}$, namely a sequence of functions $\{g_{j}\}$ such that
$ \lVert g_{j } \rVert _{L_v^{p,\kappa}(bD)} \simeq 1$,  and for a nonzero $ \phi$,   we have $\lVert \phi - [b, \EuScript C] g_{j}\rVert
 _{L_v^{p,\kappa}(bD) } <2^{-j}$.  }

Suppose that $b\notin {\rm VMO}(bD)$, then there exist $\delta_0>0$ and a sequence $\{B_j\}_{j=1}^\infty:=\{B_{r_j}(z_j)\}_{j=1}^\infty$ of balls such that
\begin{align}\label{delta0}
{1\over \lambda(B_j)} \int_{B_j} |b(z)-b_{B_j}|d\lambda(z)\geq \delta_0.
\end{align}

Without lost of generality, we assume that for all $j$, $r_j<\gamma_0$, where $\gamma_0$ is
as in Lemma \ref{nc}.

Now choose a subsequence $\{B_{j_i}\}$ of $\{B_j\}$ such that
\begin{align}\label{ratio1}
r_{j_{i+1}} \leq{1\over 4c_\lambda} r_{j_{i}},
\end{align}
where $c_\lambda$ is the constant as in \eqref{lambdab}.

For the sake of simplicity we drop the subscript $i$, i.e., we still denote $\{B_{j_i}\}$ by $\{B_{j}\}$.

Following the steps in sufficiency proof of Theorem 1.1 (2) in \cite{DLLWW}, for each such $B_j$, we can choose a corresponding ball  $\tilde B_j$  and the corresponding disjoint subsets $E_{j,1}, E_{j,2}\subset B_j$ , $\widetilde F_{j,1}, \widetilde F_{j,2}\subset\tilde B_j$, such that $B_j=E_{j,1}\cup E_{j,2}$,
\begin{align}\label{Fj1}
\lambda(\widetilde F_{j,i}) \geq{\lambda(\tilde B_j)\over 4},\quad i=1,2.
\end{align}
and
\begin{align*}
{1\over \lambda(B_j)}\int_{E_{j,1}}\big|b(z)-m_b(\tilde B_j)\big|d\lambda(z)
&\lesssim
{1\over \lambda(B_j)}\int_{E_{j,1}}\left|[b, \EuScript C](\chi_{\widetilde F_{j,1}})(z)\right|d\lambda(z).
\end{align*}
\color{black}

Next, by using H\"older's
inequality and $v\in A_p$ we further have
\begin{align*}
&{1\over \lambda(B_j)}\int_{E_{j,1}}\big|b(z)-m_b(\tilde B_j)\big|d\lambda(z)\\
&\lesssim
{1\over \lambda(B_j) } \left(\int_{E_{j,1}}v^{-{p'\over p}}d\lambda(z)\right)^{1\over p'}\bigg( \int_{E_{j,1}}\big|[b, \EuScript C](\chi_{\widetilde F_{j,1}})(z)\big|^pv(z)d\lambda(z) \bigg)^{1\over p}\\
&\lesssim
{1\over \lambda(B_j) }\lambda(B_j) v(B_j)^{-{1\over p}}v(B_j)^{{\kappa\over p}}\|[b, \EuScript C](\chi_{\widetilde F_{j,1}})\|_{L_v^{p,\kappa}(bD)}\\
&\lesssim \|[b, \EuScript C](f_j)\|_{L_v^{p,\kappa}(bD)},
\end{align*}
where in the above inequalities
we denote
$$ f_j := v(B_j)^{{\kappa-1\over p}}\chi_{\widetilde F_{j,1}} .$$

Thus, combining the above estimates we have that
$$ 0<\delta_0  \lesssim \|[b, \EuScript C](f_j)\|_{L_v^{p,\kappa}(bD)}.$$

Note that
\begin{align*}
\|f_j\|_{L_v^{p,\kappa}(bD)}&=v(B_j)^{{\kappa-1\over p}}\sup_{B}
\left\{\frac{1}{[v(B)]^\kappa}\int_{B}|\chi_{\widetilde F_{j,1}}(z)|^pv(z)\,d\lambda(z)\right\}^{1/p}\\
&=v(B_j)^{{\kappa-1\over p}}\sup_{B}
\left\{\frac{v(B\bigcap \widetilde F_{j,1})}{[v(B)]^\kappa} \right\}^{1/p}.
\end{align*}
Since $v\in A_p$, it follows that there exist positive constants $C_1,C_2$ and $\sigma\in (0,1)$ such that for any measurable set $E\subset B$,
$$\Big({\lambda(E)\over \lambda(B)}\Big)^p\leq C_1{v(E)\over v(B)}\leq C_2\Big({\lambda(E)\over \lambda(B)}\Big)^{\sigma}.$$
Combining this and \eqref{Fj1}, we obtain that
$$\sup_{B}
\left\{\frac{v(B\cap \widetilde F_{j,1})}{[v(B)]^\kappa} \right\}^{1/p}\leq \sup_{B}
\left\{ v(B\cap \widetilde F_{j,1})^{1-\kappa}  \right\}^{1/p}\leq   v( \widetilde F_{j,1})^{1-\kappa\over p}\leq   v( \widetilde B_{j })^{1-\kappa\over p}\thickapprox v(   B_{j })^{1-\kappa\over p}   $$
and
 $$\sup_{B}
\left\{\frac{v(B\cap \widetilde F_{j,1})}{[v(B)]^\kappa} \right\}^{1/p}\geq \left\{\frac{v(  \widetilde F_{j,1})}{[v(\widetilde B_{j })]^\kappa} \right\}^{1/p} \geq   v( \widetilde B_{j })^{1-\kappa\over p}\thickapprox v(   B_{j })^{1-\kappa\over p}.   $$
This implies that
$
\|f_j\|_{L_v^{p,\kappa}(bD)} \approx 1.
$

Thus, it is direct to see that $\{f_j\}_j$ is a bounded sequence in $L_v^{p,\kappa}(bD)$ with a uniform $L_v^{p,\kappa}(bD)$-lower bound away from zero.

Since $[b, \EuScript C]$ is compact, we obtain that the  sequence
$ \{[b, \EuScript C](f_j)\}_j $
has a convergent subsequence, denoted by
$$ \{[b, \EuScript C](f_{j_i})\}_{j_i}. $$
We denote the limit function by $g_0$, i.e.,
$$ [b, \EuScript C](f_{j_i})\to g_0 \quad{\rm in\ }  L_v^{p,\kappa}(bD), \quad{\rm as\ } i\to\infty.$$
Moreover,  $g_{0} \neq 0$.

After taking a further subsequence, labeled $g_{j}$, we have
\begin{itemize}
\item  $\lVert g_{j} \rVert _{L_v^{p,\kappa}(bD)}  \simeq 1$,
\item  $g_{j}$ are disjointly supported,
\item and $\lVert g_{0}  -  [b, \EuScript C] g_{j} \rVert _{L_v^{p,\kappa}(bD)}  < 2^{-j}$.
\end{itemize}

\textcolor{black}{
Take $a_{j } = j ^{ -\frac {p+1}{2p}}$, so that $ \{a_{j}\} \in \ell^{p} \setminus \ell^{1}$.   It is immediate that $ \gamma = \sum_{j} a_{j} g_{j} \in
L_v^{p,\kappa}(bD)$, hence $ [b, \EuScript C] \gamma \in L_v^{p,\kappa}(bD)$.  But, $ g_{0} \sum_{j} a_{j} \equiv \infty$, and yet
 \begin{align*}
\Bigl\lVert  g_{0} \sum_{j} a_{j} \Bigr\rVert _{L_v^{p,\kappa}(bD)}
& \leq \lVert  [b, \EuScript C] \gamma  \rVert _{L_v^{p,\kappa}(bD)}
+ \sum_{j} a_{j}  \lVert   g_{0} -  [b, \EuScript C] g_{j}  \rVert _{L_v^{p,\kappa}(bD)}  < \infty.
\end{align*}
This contradiction shows that $b\in {\rm VMO}(bD)$.
}

\smallskip

{\bf Necessity:}
Recall that $\EuScript C=\EuScript C^\sharp+\EuScript R$. Since the kernel $\EuScript C^\sharp$ is a standard kernel,  $[b, \EuScript C^\sharp]$ is compact on $L_v^{p,\kappa}(bD)$. Therefore, we only need to show that $[b, \EuScript R]$ is also compact on
$L_v^{p,\kappa}(bD)$.

From Lemma \ref {lemma 1 KL2}, for any $\xi>0$, there exists $b_\xi\in BUC(bD)$ such that
$\|b-b_\xi\|_*<\xi$. Then
by Theorem \ref{cauchy}, we have
\begin{align*}
\|[b, \EuScript R]f-[b_\xi, \EuScript R]f\|_{L_v^{p,\kappa}(bD)}\leq C_p\|f\|_{L_v^{p,\kappa}(bD)}\|b-b_\xi\|_*
<\xi C_p\|f\|_{L_v^{p,\kappa}(bD)}.
\end{align*}
Thus, to prove that $[b, \EuScript R]$ is compact on $L_v^{p,\kappa}(bD)$, it suffices to prove that $[b_\xi, \EuScript R]$ is compact on $L_v^{p,\kappa}(bD)$.
By Lemma \ref {lemma 1 KL2} and \eqref {f_eta}, without loss of generality, we may assume that $b\in BUC(bD)$ and \eqref {f_eta 1} holds.
By Lemma \ref {lemma 2 KL2}, it suffices to prove that for any fixed $\eta$ satisfying $0<\eta \ll 1$,  $[b, \EuScript R^\eta]$ is compact on $L_v^{p,\kappa}(bD)$.

Since $R(w,z)$ is continuous on  $bD\times bD\backslash \{(z,z): z\in bD\}$, for any $f\in L_v^{p,\kappa}(bD)$, we  see that $[b, \EuScript R^\eta]f$ is continuous on $bD$.
\color{black}
To conclude the proof, we now argue that the image of the unit ball of $L_v^{p,\kappa}(bD)$ under the commutator
$[b, \EuScript R^\eta] $ is an equicontinuous family.  Compactness follows from the  Ascoli--Arzela theorem.
\color{black}

It remains to prove equicontinuity.  For any $z,w\in bD$ with ${\tt d}(w,z)<1$, we have
\begin{align*}
&[b, \EuScript R^\eta]f(z)-[b, \EuScript R^\eta]f(w)\\
&=b(z)\int_{bD}R^\eta(u,z)f(u)d\lambda(u)-\int_{bD}R^\eta(u,z)b(u)f(u)d\lambda(u)\\
&\quad -b(w)\int_{bD}R^\eta(u,w)f(u)d\lambda(u)+\int_{bD}R^\eta(u,w)b(u)f(u)d\lambda(u)\\
&=(b(z)-b(w))\int_{bD}R^\eta(u,z)f(u)d\lambda(u)+b(w)\int_{bD}\left(R^\eta(u,z)- R^\eta(u,w)\right)f(u)d\lambda(u)\\
&\quad+\int_{bD}\left(R^\eta(u,w)-R^\eta(u,z)\right)b(u)f(u)d\lambda(u)\\
&=(b(z)-b(w))\int_{bD}R^\eta(u,z)f(u)d\lambda(u)
+\int_{bD}\left(R^\eta(u,w)-R^\eta(u,z)\right)\left(b(u)-b(w)\right)f(u)d\lambda(u)\\
&=:I(z,w)+I\!I(z,w).
\end{align*}

For $I(z,w)$, by H\"older's inequality, we have
\begin{align*}
&|I(z,w)|\\
&=\left|(b(z)-b(w))\right| \left|\int_{bD}R^\eta(u,z)f(u)d\lambda(u)\right|\\
&\leq c\left|(b(z)-b(w))\right|\bigg(\int_{bD}\left|R^\eta(u,z)\right|^{p'}v^{-{p'\over p}}d\lambda(u) \bigg)^{1\over p'}v(bD)^{{\kappa\over p}}\left\{\frac{1}{[v(bD)]^\kappa}\int_{bD}|f(z)|^pv(z)\,d\lambda(z)\right\}^{1/p}\\
&\leq {c \lambda(bD)v(bD)^{{\kappa-1\over p}}\|f\|_{L_v^{p,\kappa}(bD)}{\tt d}(w,z)^\epsilon},
\end{align*}
 {where the last inequality is due to the fact that $v\in A_p, R^\eta(u,z)\in C(bD\times bD)$ and $bD$ is bounded.}

Since $b$ is bounded, if we let ${\tt d}(w,z)<{\eta\over c\cdot  c_R}$, by a discussion similar to \cite[p. 645]{KL2}, we can obtain that
\begin{align*}
|I\!I(z,w)|
&=\left| \int_{bD}\left(R^\eta(u,w)-R^\eta(u,z)\right)\left(b(u)-b(w)\right)f(u)d\lambda(u) \right|\\
 &\leq
c \|b\|_{L^\infty(bD)}\int_{bD\setminus B_{\eta\over c}(z)}{{\tt d}(w,z)\over {\tt d}(u,z)^{2n}} |f(u)|d\lambda(u)\\
&\leq
c \|b\|_{L^\infty(bD)}v(bD)^{{\kappa\over p}}\|f\|_{L_v^{p,\kappa}(bD)}{\tt d}(w,z)
\left(\int_{bD\setminus B_{\eta\over c}(z)}{1\over {\tt d}(u,z)^{2np'}}v^{-{p'\over p}}d\lambda(u)\right)^{1\over p'}\\
&\leq
c \|b\|_{L^\infty(bD)}v(bD)^{{\kappa\over p}}\|f\|_{L_v^{p,\kappa}(bD)}{\tt d}(w,z)\big({\eta\over c} \big)^{-2n}\lambda(bD)v(bD)^{-{1\over p}}\\
&\leq
c_{\eta,p}\lambda(bD)v(bD)^{{\kappa-1\over p}}\|b\|_{L^\infty(bD)}\|f\|_{L_v^{p,\kappa}(bD)}{\tt d}(w,z).
\end{align*}

Therefore, $\{[b, \EuScript R^\eta](\mathcal U)\}$ is an equicontinuous family, where $\mathcal U$ is the unit ball in $L_v^{p,\kappa}(bD)$.
This finishes the proof of  Theorem \ref{vmo}.
\end{proof}

\section{Commutator of Cauchy-Leray integral for strongly $\mathbb C$-linearly convex domains with minimal smoothness}
\setcounter{equation}{0}

In this section, we focus on the bounded domain $D\subset \mathbb C^n$ which is  strongly $\mathbb C$-linearly convex and whose boundary satisfies the minimal regularity condition of class $C^{1,1}$ \cite{LS2014}.

Suppose $D$ is a bounded domain in $\C^n$ with defining function $\rho$
satisfying

1) $D$ is of class $C^{1,1}$, i.e., the first derivatives of its defining function $\rho$ are Lipschitz, and $|\nabla \rho(w)|>0$ whenever $w\in\{w:\rho(w)=0\}=bD$;

2) $D$ is strongly $\C$-linearly convex, i.e., $D$ is a bounded domain of $C^1$, and at any boundary point it satisfies either of the following two equivalent conditions
\begin{align*}
 |\Delta(w,z)|&\geq c|w-z|^2, \\
 d_E\big(z,w+T_w^{\mathbb C}\big)&\geq \tilde c|w-z|^2, 
\end{align*}
for some $c,\tilde c>0$, where
\begin{align}\label{delta}
\Delta(w,z)=\langle\partial \rho(w), w-z\rangle,
\end{align}
and $d_E(z,w+T_w^{\mathbb C})$ denotes the Euclidean distance from $z$ to the affine subspace $w+T_w^{\mathbb C}$. Note that
$T_w^{\mathbb C}:=\{v:\langle\partial \rho(w),v\rangle=0\}$ is the complex tangent space referred to the origin, $w+T_w^{\mathbb C}$ is its geometric realization as an affine space tangent to $bD$ at $w$.


On $bD$ there is a quasi-distance ${\tt d}$, which is defined as
$$ {\tt d}(w,z)= |\Delta(w,z)|^{1\over 2} = |\langle \partial\rho,w-z \rangle|^{1\over2},\quad w,z\in bD. $$

%
%

Let $d\lambda$ be the Leray--Levi measure $d\lambda$ on $bD$  (c.f. \cite{LS2014}).
According to \cite[Proposition 3.4]{LS2014},  $d\lambda$ is also equivalent to the induced Lebesgue measure $d\sigma$ on $bD$ in the following sense:
$$ d\lambda(w) =\tilde\Lambda(w)d\sigma(w)\quad {\rm \ for\ }\sigma {\rm\ a.e.\ } w\in bD,$$
and there are two strictly positive constants $c_1$ and $c_2$ so that
$$ c_1\leq \tilde\Lambda(w)\leq c_2 \quad {\rm \ for\ }\sigma {\rm\ a.e.\ } w\in bD.$$

We also denote by $ B_r(w) =\{ z\in bD:\ {\tt d}(w,z)<r \}$  the  boundary balls determined via the quasidistance ${\tt d }$.
By \cite[Proposition 3.5]{LS2014}, we also have
\begin{align}\label{lambdab-1}
\lambda\big(B_r(w) \big)\approx r^{2n},\quad 0<r\leq 1.
\end{align}

The Cauchy--Leray integral of a suitable function $f$ on $bD$, denoted ${\bf C}(f)$, is formally defined by
$$  {\bf C}(f)(z) =\int_{bD} {f(w)\over \Delta(w,z)^n}d\lambda(w),\quad z\in D. $$

When restricting $z$ to the boundary $bD$, we have the Cauchy--Leray transform $f\mapsto \mathcal C(f)$, defined as
$$  {\mathcal C}(f)(z) =\int_{bD} {f(w)\over \Delta(w,z)^n}d\lambda(w),\quad z\in bD, $$
where the function $f$ satisfies the H\"older-like condition
$$ |f(w_1)-f(w_2)|\lesssim {\tt d}(w_1,w_2)^\alpha,\quad 
 w_1,w_2\in bD, $$
  for some $0<\alpha\leq 1$.


Take $K(w,z)$ to be the function defined for $w,z\in bD$, with $w\not= z$, by
$$K(w,z) = {1\over \Delta(w,z)^n}.$$
This function is the ``kernel" of the operator  $\mathcal C$, in the sense that
$$  \mathcal C(f)(z)= \int_{bD} K(w,z)f(w)d\lambda(w), $$
whenever $z$ lies outside of the support of $f$ and $f$ satisfies the H\"older-like condition for some $\alpha$.
The size and regularity estimates that are relevant for us are:
\begin{align}\label{ck}
& |K(w,z)|\lesssim {1\over{\tt d}(w,z)^{2n}};\nonumber\\
& |K(w,z) - K(w',z)|\lesssim  { {\tt d}(w,w')\over  {\tt d}(w,z)^{2n+1} },\quad {\rm if}\ {\tt d}(w,z)\geq c_K{\tt d}(w,w');\\
& |K(w,z) - K(w,z')|\lesssim  { {\tt d}(z,z')\over  {\tt d}(w,z)^{2n+1} },\quad {\rm if}\ {\tt d}(w,z)\geq c_K {\tt d}(z,z'),\nonumber
\end{align}
for an appropriate constant $c_K>0$. Moreover, for the size estimates we actually have
\begin{align}\label{kd}
 |K(w,z)|= {1\over{\tt d}(w,z)^{2n}}.
 \end{align}

We need the $L^p(bD)$ boundedness of Cauchy--Leray transform in \cite[Theorem 5.1]{LS2014}.

\begin{lemma}
The Cauchy--Leray transform $f\mapsto \mathcal C(f)$, initially defined for functions $f$ that satisfy the H\"older-like condition for some $\alpha$, extends to a bounded linear operator on $L^p(bD)$ for $1<p<\infty$.
\end{lemma}

\begin{proof}[Proof of Theorem \ref {weight1} and \ref {Cauchy-Leray}]
We point out that the proof of Theorem \ref {weight1} follows from the proof of Theorem 1.5 in \cite{DLLW}, and the proof of Theorem \ref {Cauchy-Leray} follows from the proof of Theorem \ref {cauchy} and Theorem \ref{vmo}. In fact, these are simpler, since the operator $\mathcal C$ is a Calder\'on--Zygmund operator.
\end{proof}

\color{black}

\bigskip

{\bf Acknowledgement:} This work was supported by Natural Science Foundation
of China (Grant Nos. 11671185, 11701250 and 11771195), Natural Science Foundation of Shandong
Province (Grant Nos. ZR2018LA002 and ZR2019YQ04) and the State Scholarship Fund of China (No. 201908440061).

\bibliographystyle{amsplain}

\end{document}